\documentclass[a4paper]{amsart}
\usepackage{graphicx}
\usepackage{amssymb}
\usepackage{amsmath}
\usepackage{amsthm}
\usepackage{amscd}
\usepackage[all,2cell]{xy}

\UseAllTwocells \SilentMatrices
\newtheorem{thm}{Theorem}[section]

\newtheorem{cor}[thm]{Corollary}
\newtheorem{lem}[thm]{Lemma}
\newtheorem{exm}[thm]{Example}

\newtheorem{prop}[thm]{Proposition}
\theoremstyle{definition}
\newtheorem{defn}[thm]{Definition}
\theoremstyle{remark}
\newtheorem{rem}[thm]{\bf Remark}
\numberwithin{equation}{section}

\begin{document}
\title[A note on morphisms determined by objects]
{A note on morphisms determined by objects}

\author[Xiao-Wu Chen, Jue Le] {Xiao-Wu Chen, Jue Le}

\thanks{The authors are supported by National Natural Science Foundation of China (No.s 11101388 and 11201446) and NCET-12-0507.}
\date{\today}

\thanks{E-mail:
xwchen$\symbol{64}$mail.ustc.edu.cn, juele$\symbol{64}$ustc.edu.cn}

\maketitle

\dedicatory{}%
\commby{}%

\begin{abstract}
We prove that a Hom-finite additive category having determined morphisms on both sides is a dualizing variety. This complements a result by Krause. We prove that in a Hom-finite abelian category  having Serre duality, a morphism is right determined by some object if and only if it is an epimorphism. We give a characterization to abelian categories having Serre duality via determined morphisms.
\end{abstract}

\section{Introduction}

 Let $\mathcal{C}$ be an additive category which is skeletally small, that is, the iso-classes of objects form a set. Let $Y$ be an object in $\mathcal{C}$. Let  $\alpha\colon X\rightarrow Y$ be a morphism and $C$ be an object in $\mathcal{C}$. Set $\Gamma(C)={\rm End}_\mathcal{C}(C)^{\rm op}$ to be the opposite ring of the endomorphism ring of $C$.  Consider the induced map ${\rm Hom}_\mathcal{C}(C, \alpha)\colon {\rm Hom}_\mathcal{C}(C, X)\rightarrow {\rm Hom}_\mathcal{C}(C, Y)$ between left $\Gamma(C)$-modules. The image ${\rm Im}\; {\rm Hom}_\mathcal{C}(C, \alpha)$ is a $\Gamma(C)$-submodule of ${\rm Hom}_\mathcal{C}(C, Y)$.

 Recall from \cite{Aus1, Aus2} that $\alpha \colon X\rightarrow Y$ is \emph{right determined} by $C$ provided that for any morphism $t\colon T\rightarrow Y$, ${\rm Im}\; {\rm Hom}_\mathcal{C}(C, t)\subseteq {\rm Im}\; {\rm Hom}_\mathcal{C}(C, \alpha)$ implies that $t$ factors through $\alpha$, that is, there exists a morphism $t'\colon T\rightarrow X$ with $t=\alpha\circ t'$.

 For a $\Gamma(C)$-submodule $H$ of ${\rm Hom}_\mathcal{C}(C, Y)$, we say that the pair $(C, H)$ is \emph{right represented} by a morphism $\alpha\colon X\rightarrow Y$ provided that $\alpha$ is right determined by $C$ and $H={\rm Im}\; {\rm Hom}_\mathcal{C}(C, \alpha)$.

The following notion is essentially contained in \cite[Definition 2.6]{Kr}.

 \begin{defn}\label{defn:1}
An object $Y$ in $\mathcal{C}$ is \emph{right classified} provided that the following hold:
\begin{enumerate}
 \item[(RC1)] each morphism $\alpha\colon X\rightarrow Y$ ending at $Y$ is right determined by some object $C$;
 \item[(RC2)] for any object $C$ and  any $\Gamma(C)$-submodule $H$ of ${\rm Hom}_\mathcal{C}(C, Y)$, the pair $(C, H)$ is  right represented by some morphism $\alpha \colon X\rightarrow Y$.
 \end{enumerate}
 The additive category $\mathcal{C}$ is said to  \emph{have right determined morphisms} if each object is right classified. \hfill $\square$
 \end{defn}

 Let us justify this terminology.  Two morphisms $\alpha_1\colon X_1\rightarrow Y$ and $\alpha_2\colon X_2\rightarrow Y$ are \emph{right equivalent} if $\alpha_1$ factors through $\alpha_2$ and $\alpha_2$ factors through $\alpha_1$. The corresponding right equivalence class is denoted by $[\alpha_1\rangle=[\alpha_2\rangle$. Following \cite{Rin2}, we denote by $[\longrightarrow Y\rangle$ the set of right equivalence classes of morphisms ending at $Y$. It is indeed a set, since $\mathcal{C}$ is skeletally small.

  If two morphisms $\alpha_1$ and $\alpha_2$ are right equivalent, then $\alpha_1$ is right determined by $C$ if and only if so is $\alpha_2$. So it makes sense to say that the class $[\alpha_1\rangle$ is right determined by $C$. We denote by $^C[\longrightarrow Y\rangle$ the subset of $[\longrightarrow Y\rangle$ formed by classes which are right determined by $C$. Then (RC1) is equivalent to
  \begin{align}\label{equ:1}
  [\longrightarrow Y\rangle=\bigcup {^C[\longrightarrow Y\rangle},
  \end{align}
   where $C$ runs over all objects in $\mathcal{C}$.

  We denote by ${\rm Sub}\; {\rm Hom}_\mathcal{C}(C, Y)$ the set of $\Gamma(C)$-submodules of ${\rm Hom}_\mathcal{C}(C, Y)$. The following map is well-defined
  $$\eta_{C, Y}\colon [\longrightarrow Y\rangle \longrightarrow {\rm Sub}\; {\rm Hom}_\mathcal{C}(C, Y), \quad [\alpha \rangle \mapsto {\rm Im}\; {\rm Hom}_\mathcal{C}(C, \alpha).$$
  The restriction of $\eta_{C, Y}$ on  $^C[\longrightarrow Y\rangle$ is injective by the following direct consequence of the definition.

  \begin{lem}\label{lem:REHom}
  Let $\alpha_1\colon X_1\rightarrow Y$ and $\alpha_2\colon X_2\rightarrow Y$ be two morphisms that are right determined by some object $C$. Then $\alpha_1$ is right equivalent to $\alpha_2$ if and only if ${\rm Im}\; {\rm Hom}_\mathcal{C}(C, \alpha_1)={\rm Im}\; {\rm Hom}_\mathcal{C}(C, \alpha_2)$. \hfill $\square$
  \end{lem}

  Then (RC2) is equivalent to the surjectivity of this restriction. In other words, (RC2) is equivalent to the bijection
  \begin{align}\label{equ:2}
  ^C[\longrightarrow Y\rangle \stackrel{\sim}\longrightarrow {\rm Sub}\; {\rm Hom}_\mathcal{C}(C, Y), \quad [\alpha\rangle \mapsto {\rm Im}\; {\rm Hom}_\mathcal{C}(C, \alpha).
  \end{align}
This bijection is known as the \emph{Auslander bijection} at $Y$; see \cite{Rin2}.

 In summary, an object $Y$ is right classified if and only if  (\ref{equ:1}) and (\ref{equ:2}) hold.  In this case,  all morphisms ending at $Y$ are classified by the pairs $(C, H)$ of objects $C$ and $\Gamma(C)$-submodules $H$ of ${\rm Hom}_\mathcal{C}(C, Y)$.

The dual notion is as follows.

\begin{defn}
An object $Y$ in $\mathcal{C}$ is \emph{left classified} if it is right classified as an object in the opposite category $\mathcal{C}^{\rm op}$.
The additive category $\mathcal{C}$ is said to  \emph{have left determined morphisms} if each object is left classified.

 The additive category $\mathcal{C}$ \emph{has determined morphisms} if it has both right and left determined morphisms. \hfill $\square$
\end{defn}

One of the fundamental results is that the category $A\mbox{-mod}$ of finitely generated modules over an artin algebra $A$
 has determined morphisms; for example, see \cite{ARS, Rin1}. This result is extended to dualizing $k$-varieties for a commutative artinian ring $k$ in \cite{Kr}. We prove that the converse is true. More precisely, if an additive category $\mathcal{C}$ is $k$-linear which is Hom-finite and has determined morphisms, then it is a dualizing $k$-variety; see Proposition \ref{prop:Dua}.  When the category $\mathcal{C}$ is abelian having  Serre duality, we prove that a morphism is right determined by some object if and only if it is an epimorphism, and dually,  a morphism is left determined by some object if and only if it is a monomorphism; see Remark \ref{rem:1}(1). Indeed, we give a characterization to abelian categories having Serre duality via determined morphisms; see Theorem \ref{thm:1}. In particular, we point out that a non-trivial abelian category having Serre duality is not a dualizing $k$-variety; see Remark \ref{rem:1}(2).

\section{Categories having determined morphisms}

Let $k$ be a commutative artinian ring with a unit, and  let $\mathcal{C}$ be a $k$-linear additive category. We assume that $\mathcal{C}$ is \emph{Hom-finite}, that is, the $k$-module ${\rm Hom}_\mathcal{C}(X, Y)$ is finitely generated for any objects $X$ and $Y$ in $\mathcal{C}$. We suppose further that $\mathcal{C}$ is skeletally small, meaning that the iso-classes of objects in $\mathcal{C}$ form a set.

We denote by $k\mbox{-mod}$ the abelian category of finitely generated $k$-modules. Let $E$ be the minimal injective cogenerator of $k$. Then we have the duality $D={\rm Hom}_k(-, E)\colon k\mbox{-mod}\rightarrow k\mbox{-mod}$ with $D^2\simeq {\rm Id}_{k\mbox{-}{\rm mod}}$.

Denote by $(\mathcal{C}, k\mbox{-mod})$ the abelian category of $k$-linear functors from $\mathcal{C}$ to $k\mbox{-mod}$. The above duality $D$ induces a duality
\begin{align}\label{equ:du}
D\colon (\mathcal{C}, k\mbox{-mod})\stackrel{\sim}\longrightarrow (\mathcal{C}^{\rm op}, k\mbox{-mod})^{\rm op}
\end{align}
 sending a functor $F$ to $DF$. Here, $\mathcal{C}^{\rm op}$ denotes the opposite category of $\mathcal{C}$.

Recall the Yoneda embedding $\mathcal{C}\rightarrow (\mathcal{C}^{\rm op}, k\mbox{-mod})$  sending $X$ to ${\rm Hom}_\mathcal{C}(-, X)$. Then we have the following natural isomorphisms
\begin{align}\label{equ:Yoneda1}
{\rm Hom}_{(\mathcal{C}^{\rm op}, k\mbox{-}{\rm mod})}({\rm Hom}_\mathcal{C}(-, C'), F)\stackrel{\sim}\longrightarrow F(C')\stackrel{\sim}\longrightarrow {\rm Hom}_{\Gamma(C)}({\rm Hom}_\mathcal{C}(C, C'), F(C))
\end{align}
for any $F\in (\mathcal{C}^{\rm op}, k\mbox{-mod})$ and $C, C'\in \mathcal{C}$ with $C'\in {\rm add}\; C$. Here, ${\rm add}\; C$ denotes the full subcategory formed by direct summands of finite direct sums of $C$, and $\Gamma(C)={\rm End}_\mathcal{C}(C)^{\rm op}$. This composite sends a morphism $\eta$ to $\eta_C$. The left isomorphism is known as the Yoneda Lemma, from which it follows that ${\rm Hom}_\mathcal{C}(-, C')$ is a projective object in $(\mathcal{C}^{\rm op}, k\mbox{-mod})$.

By (\ref{equ:Yoneda1}) and the duality (\ref{equ:du}), we have the following natural isomorphisms
\begin{align}\label{equ:Yoneda}
{\rm Hom}_{(\mathcal{C}^{\rm op}, k\mbox{-}{\rm mod})}(F, D{\rm Hom}_\mathcal{C}(C', -))\stackrel{\sim}\longrightarrow DF(C')\stackrel{\sim}\longrightarrow {\rm Hom}_{\Gamma(C)}(F(C),D{\rm Hom}_\mathcal{C}(C', C))
\end{align}
for any $F\in (\mathcal{C}^{\rm op}, k\mbox{-mod})$ and $C, C'\in \mathcal{C}$ with $C'\in {\rm add}\; C$. The composite sends $\eta$ to $\eta_C$.

 A functor $F\colon \mathcal{C}^{\rm op}\rightarrow k\mbox{-mod}$ is \emph{finitely generated} if there is an epimorphism ${\rm Hom}_\mathcal{C}(-, Y)\rightarrow F$ for some object $Y$; it is \emph{finitely cogenerated} if there is a monomorphism $F\rightarrow D{\rm Hom}_\mathcal{C}(C, -)$ for some object $C$, or equivalently, its dual $DF$ is finitely generated. The functor $F\colon \mathcal{C}^{\rm op}\rightarrow k\mbox{-mod}$ is \emph{finitely presented} if there is an exact sequence of functors
 $${\rm Hom}_\mathcal{C}(-, X)\longrightarrow {\rm Hom}_\mathcal{C}(-, Y)\longrightarrow F\longrightarrow 0.$$
 We denote by  ${\rm fp}(\mathcal{C})$ the full subcategory of $(\mathcal{C}^{\rm op}, k\mbox{-mod})$ consisting of finitely presented functors.

Following \cite[Section 2]{AR}, the category $\mathcal{C}$ is a \emph{dualizing $k$-variety} provided that  any functor $F\colon \mathcal{C}^{\rm op}\rightarrow k\mbox{-mod}$ is finitely presented if and only if so is its dual $DF$. In this case, the subcategory ${\rm fp}(\mathcal{C})\subseteq (\mathcal{C}^{\rm op}, k\mbox{-mod})$ is \emph{exact abelian}, meaning that it is closed under kernels, cokernels and images; consult \cite[Theorem 2.4]{AR}. We mention that by definition $\mathcal{C}$ is a dualizing $k$-variety if and only if so is $\mathcal{C}^{\rm op}$.

The aim of this section is to prove the following result. The implication ``$(3)\Rightarrow (1)$" is given in \cite[Corollary 2.13]{Kr}. We mention that the implication ``$(1)\Rightarrow (3)$" is somewhat implicit in the argument in \cite[Sections 3 and 5]{Kr}. Hence, Proposition \ref{prop:Dua} is simply missed in \cite{Kr}. Here we make this result explicit.

\begin{prop}\label{prop:Dua}
Let $\mathcal{C}$ be a Hom-finite $k$-linear additive category which is skeletally small. Then the following statements are equivalent:
\begin{enumerate}
\item the category $\mathcal{C}$ has determined morphisms;
\item for any functor $F$ in $(\mathcal{C}, k\mbox{-}{\rm mod})$ or $(\mathcal{C}^{\rm op}, k\mbox{-}{\rm mod})$, $F$ is finite presented if and only if $F$ is finitely generated and finitely cogenerated;
\item the category $\mathcal{C}$ is a dualizing $k$-variety.
\end{enumerate}
\end{prop}

\begin{proof}
The equivalence between $(1)$ and $(2)$ follows from Corollary \ref{cor:RD} and its dual, while the equivalence between $(2)$ and (3) follows from Lemma \ref{lem:Dual}.
\end{proof}

The following result is well-known and implicit in \cite[Proposition 3.1]{AR}.

\begin{lem}\label{lem:Dual}
Let $\mathcal{C}$ be as above. Then $\mathcal{C}$ is a dualizing $k$-variety if and only if the following two conditions hold:
\begin{itemize}
\item[(1)] any functor $F\colon \mathcal{C}^{\rm op}\rightarrow k\mbox{-{\rm mod}}$ is finitely presented $\Longleftrightarrow$ it is finitely generated and finitely cogenerated;
\item[(2)] any functor $F\colon \mathcal{C}\rightarrow k\mbox{-{\rm mod}}$ is finitely presented $\Longleftrightarrow$ it is finitely generated and finitely cogenerated;
\end{itemize}
\end{lem}

\begin{proof}
We observe that the duality (\ref{equ:du}) preserves the functors that are both finitely generated and finitely cogenerated. Then the ``if" part follows.

For the ``only if" part, we assume that $\mathcal{C}$ is a dualizing $k$-variety and we only prove (1). Indeed, if $F$ is finitely presented, then $DF$ is finitely presented, in particular, $DF$ is finitely generated. Hence $F$ is finitely cogenerated. This yields the direction ``$\Longrightarrow$". Conversely, if $F$ is finitely generated and finitely cogenerated, then $F$ is the image of some morphism $\theta\colon {\rm Hom}_\mathcal{C}(-, X)\rightarrow D{\rm Hom}_\mathcal{C}(C, -)$. The morphism $\theta$ is in the category ${\rm fp}(\mathcal{C})$. Recall that for a dualizing $k$-variety $\mathcal{C}$, the subcategory ${\rm fp}(\mathcal{C})\subseteq (\mathcal{C}^{\rm op}, k\mbox{-mod})$  is closed under images. We infer that $F$ is finitely presented.
\end{proof}

Let $\alpha\colon X\rightarrow Y$ be a morphism in $\mathcal{C}$. Then we define a finitely presented functor $F^\alpha$ by the exact sequence 
$${\rm Hom}_\mathcal{C}(-, X)\stackrel{{\rm Hom}_\mathcal{C}(-, \alpha)}\longrightarrow {\rm Hom}_\mathcal{C}(-, Y)\longrightarrow F^\alpha\rightarrow 0.$$
By the Yoneda Lemma, every finitely presented functor arises in this way.

The following result is contained in \cite[Proposition 5.2]{Kr}. We give a proof for completeness.

\begin{lem}\label{lem:RD}
The morphism $\alpha$ is right determined by an object $C$ if and only if there is a monomorphism $F^\alpha\rightarrow D{\rm Hom}_\mathcal{C}(C', -)$ for some $C'\in {\rm add}\; C$.
\end{lem}

\begin{proof}
For the ``only if" part, we assume that $\alpha\colon X\rightarrow Y$ is right determined by $C$. Take an exact sequence of $\Gamma(C)$-modules for some $C'\in {\rm add}\; C$
$${\rm Hom}_\mathcal{C}(C, X)\stackrel{{\rm Hom}_\mathcal{C}(C, \alpha)}\longrightarrow {\rm Hom}_\mathcal{C}(C, Y)\stackrel{\theta_C}\longrightarrow D{\rm Hom}_\mathcal{C}(C', C).$$
Indeed, we may take an injective map ${\rm Cok}\; {\rm Hom}_\mathcal{C}(C, \alpha)\hookrightarrow D{\rm Hom}_\mathcal{C}(C', C)$ for some $C'\in {\rm add}\; C$; here, we use the fact that $D{\rm Hom}_\mathcal{C}(C, C)$ is an injective cogenerator as a $\Gamma(C)$-module. By the isomorphism (\ref{equ:Yoneda}), the map $\theta_C$ induces a morphism $\theta\colon {\rm Hom}_\mathcal{C}(-, Y)\rightarrow D{\rm Hom}_\mathcal{C}(C', -)$.
We claim that the following sequence of functors is exact, which yields the required monomorphism.
\begin{align}\label{equ:FE}
{\rm Hom}_\mathcal{C}(-, X) \stackrel{{\rm Hom}_\mathcal{C}(-, \alpha)}\longrightarrow {\rm Hom}_\mathcal{C}(-, Y) \stackrel{\theta}\longrightarrow D{\rm Hom}_\mathcal{C}(C', -)
\end{align}
The composite  is zero by the isomorphism (\ref{equ:Yoneda}). Let $t\colon T\rightarrow Y$ be any morphism in ${\rm Ker}\; \theta_T$. For any morphism $\psi\colon C\rightarrow T$, the morphism $t\circ \psi$ lies in ${\rm Ker}\; \theta_C$ by the naturalness of $\theta$, and thus in ${\rm Im}\; {\rm Hom}_\mathcal{C}(C, \alpha)$. In other words, ${\rm Im}\; {\rm Hom}_\mathcal{C}(C, t)\subseteq {\rm Im}\; {\rm Hom}_\mathcal{C}(C, \alpha)$. Since $\alpha$ is right determined by $C$, we infer that $t$ factors through $\alpha$. This proves that the above sequence is exact.

For the ``if" part,  we  may assume that we have an exact sequence as (\ref{equ:FE}). Take any morphism $t\colon T\rightarrow Y$ such that ${\rm Im}\; {\rm Hom}_\mathcal{C}(C, t)\subseteq {\rm Im}\; {\rm Hom}_\mathcal{C}(C, \alpha)$. Hence,  $\theta_C\circ {\rm Hom}_\mathcal{C}(C, t)=0$. By the isomorphism (\ref{equ:Yoneda}), we have $\theta \circ {\rm Hom}_\mathcal{C}(-, t)=0$. Recall that ${\rm Hom}_\mathcal{C}(C, -)$ is a projective object in $(\mathcal{C}^{\rm op}, k\mbox{-mod})$. Then the exact sequence (\ref{equ:FE}) yields that ${\rm Hom}_\mathcal{C}(-, t)$ factors through ${\rm Hom}_\mathcal{C}(-, \alpha)$. By the  Yoneda Lemma, $t$ factors through $\alpha$.  Then we are done.
\end{proof}

Let $Y$ be an object. Consider a pair $(C, H)$ with $C$ an object and $H\subseteq {\rm Hom}_\mathcal{C}(C, Y)$ a $\Gamma(C)$-submodule. Recall that $D{\rm Hom}_\mathcal{C}(C, C)$ is an injective cogenerator as a $\Gamma(C)$-module.
Take an embedding of $\Gamma(C)$-modules $$ {\rm Hom}_\mathcal{C}(C, Y)/H\hookrightarrow D{\rm Hom}_\mathcal{C}(C', C)$$  for some $C'\in {\rm add}\; C$. This gives rise to  a map $\theta_C\colon {\rm Hom}_\mathcal{C}(C, Y)\rightarrow D{\rm Hom}_\mathcal{C}(C', C)$, which corresponds via (\ref{equ:Yoneda}) to a morphism $\theta \colon {\rm Hom}_\mathcal{C}(-, Y)
\rightarrow D{\rm Hom}_\mathcal{C}(C', -)$. Denote its image by $F^{(C, H)}$; it is a finitely generated and finitely cogenerated functor. Indeed, all functors in $(\mathcal{C}^{\rm op}, k\mbox{-mod})$ that are finitely generated and finitely cogenerated  arise in this way.

\begin{lem}\label{lem:CH}
The pair $(C, H)$ is right represented by some morphism if and only if the functor $F^{(C, H)}$ is finite presented.
\end{lem}

\begin{proof}
For the ``only if" part, assume that $(C, H)$ is right represented by a morphism $\alpha\colon X\rightarrow Y$; in particular, $H={\rm Im}\; {\rm Hom}_\mathcal{C}(C, \alpha)$. By the proof of  Lemma \ref{lem:RD}, the functor $F^\alpha$ is the image of the morphism $\theta\colon {\rm Hom}_\mathcal{C}(-, Y) \rightarrow D{\rm Hom}_\mathcal{C}(C', -)$. It follows that $F^{(C, H)}=F^\alpha$. In particular, it is finitely presented.

For the ``if" part, assume that $F^{(C, H)}$ is finitely presented. The kernel of the epimorphism ${\rm Hom}_\mathcal{C}(-, Y)\rightarrow F^{(C, H)}$ is finitely generated. Hence there exists a map $\alpha\colon X\rightarrow Y$ such that the following sequence is exact
\begin{align}
{\rm Hom}_\mathcal{C}(-, X) \stackrel{{\rm Hom}_\mathcal{C}(-, \alpha)}\longrightarrow {\rm Hom}_\mathcal{C}(-, Y) \stackrel{\theta}\longrightarrow D{\rm Hom}_\mathcal{C}(C', -).
\end{align}
Hence, ${\rm Im}\; {\rm Hom}_\mathcal{C}(C, \alpha)=H$ and $F^{(C, H)}\simeq F^\alpha$. By Lemma \ref{lem:RD} the map $\alpha$ is right determined by $C$. Then the pair $(C, H)$ is right represented by $\alpha$.
\end{proof}

\begin{cor}\label{cor:RC}
Let $Y$ be an object in $\mathcal{C}$. Then the following statements are equivalent:
\begin{itemize}
\item[(1)] the object $Y$ is right classified;
\item[(2)] for any quotient functor $F$ of ${\rm Hom}_\mathcal{C}(-, Y)$, $F$ is finitely presented if and only if $F$ is finitely cogenerated.
\end{itemize}
\end{cor}

\begin{proof}
 Observe that the quotient functor $F$ is finitely presented if and only if $F=F^\alpha$ for some morphism $\alpha\colon X\rightarrow Y$, and that $F$ is finitely cogenerated if and only if $F=F^{(C, H)}$ for a pair $(C, H)$. Then the result follows from Lemmas \ref{lem:RD} and \ref{lem:CH}.
\end{proof}

The following is an immediate consequence of the above result.

\begin{cor}\label{cor:RD}
Let $\mathcal{C}$ be as above. Then the following statements are equivalent:
\begin{itemize}
\item[(1)] the additive category $\mathcal{C}$ has right determined morphisms;
\item[(2)] for any functor $F$ in $(\mathcal{C}^{\rm op}, k\mbox{-}{\rm mod})$, $F$ is finitely presented if and only if $F$ is finitely generated and finitely cogenerated. \hfill $\square$
\end{itemize}
\end{cor}

\begin{exm}
{\rm Let $\mathcal{C}$ be a Hom-finite additive $k$-linear category which is skeletally small and has split idempotents. Hence, the category $\mathcal{C}$ is Krulll-Schmidt. Denote by ${\rm ind}\; \mathcal{C}$ the set of iso-classes of indecomposable objects in $\mathcal{C}$. We assume that for each object $Y$, there are only finitely many $X\in {\rm ind}\; \mathcal{C}$ such that ${\rm Hom}_\mathcal{C}(X, Y)\neq 0$, and that there exists an object $C_0$ such that ${\rm Hom}_\mathcal{C}(C_0, X)\neq 0$ for infinitely many $X\in {\rm ind}\; \mathcal{C}$. For example, the category of preprojective modules over a tame hereditary algebra satisfies this condition.

In this case, every finitely generated functor $F$ in $(\mathcal{C}^{\rm op}, k\mbox{-}{\rm mod})$ has finite length, and it follows that $F$ is finitely presented and finitely cogenerated. However,  every  finitely cogenerated functor in  $(\mathcal{C}, k\mbox{-}{\rm mod})$ has finite length, and thus the functor ${\rm Hom}_\mathcal{C}(C_0, -)$ is not finitely cogenerated. It follows from Corollary \ref{cor:RD} that $\mathcal{C}$ has right determined morphism, but does not have left determined morphisms. Indeed, by the dual of Corollary \ref{cor:RC},  an object $C$ is left classified if and only if there are only finitely many $X\in {\rm ind}\; \mathcal{C}$ such that ${\rm Hom}_\mathcal{C}(C, X)\neq 0$.}
\end{exm}

\section{Abelian categories having Serre duality}

Let $\mathcal{C}$ be a Hom-finite $k$-linear abelian category. The category $\mathcal{C}$ is said to \emph{have Serre duality} provided that there exists a $k$-linear auto-equivalence $\tau \colon \mathcal{C}\rightarrow \mathcal{C}$ with a functorial isomorphism
\begin{align}
D {\rm Ex}_\mathcal{C}^1(X, Y)\stackrel{\sim}\longrightarrow {\rm Hom}_\mathcal{C}(Y, \tau(X))
\end{align}
for any objects $X, Y$ in $\mathcal{C}$. The functor $\tau$ is called the \emph{Auslander-Reiten translation} of $\mathcal{C}$.

The following notion is modified from Definition \ref{defn:1}.

 \begin{defn}\label{defn:REC}
An object $Y$ in $\mathcal{C}$ is \emph{right epi-classified} provided that  the following hold:
\begin{enumerate}
 \item[(REC1)] each epimorphism $\alpha\colon X\rightarrow Y$ ending at $Y$ is right determined by some object $C$;
 \item[(REC2)] for any object $C$ and any  $\Gamma(C)$-submodule $H$ of ${\rm Hom}_\mathcal{C}(C, Y)$,  the pair $(C, H)$ is  right represented by some epimorphism $\alpha\colon X\rightarrow Y$.
 \end{enumerate}
If each object in $\mathcal{C}$ is right epi-classified, the abelian category $\mathcal{C}$ is said to \emph{have right determined epimorphisms}. \hfill $\square$
 \end{defn}

We observe  the following fact.

\begin{lem}\label{lem:REC}
Let $Y$ be right epi-classified. Then a morphism ending at $Y$ is right determined by some object if and only if it is an epimorphism.
\end{lem}

Consequently, if the abelian category $\mathcal{C}$ has right determined epimorphisms, a morphism is right determined by some object if and only if it is an epimorphism.

\begin{proof}
Recall that for two right equivalent maps $\alpha_1\colon X \rightarrow Y$ and $\alpha_2\colon X_2\rightarrow Y$, $\alpha_1$ is epic if and only if so is $\alpha_2$.  For the ``only if" part of this lemma, let $\alpha\colon X\rightarrow Y$ be a morphism which is right determined by some object $C$. By (REC2) the pair $(C, {\rm Im}\; {\rm Hom}_\mathcal{C}(C, \alpha))$ is right represented by an epimorphism $\alpha'\colon X'\rightarrow Y$. Lemma \ref{lem:REHom} implies that $\alpha$ and $\alpha'$ are right equivalent. It follows that $\alpha$ is epic.
\end{proof}

We denote by $[\longrightarrow Y\rangle_{\rm epi}$ the subset of $[\longrightarrow Y\rangle$ formed by epimorphisms. As in Introduction, the object $Y$ is right epi-classified implies that $^C[\longrightarrow Y\rangle={^C[\longrightarrow Y\rangle_{\rm epi}}$ and $[\longrightarrow Y\rangle_{\rm epi}=\bigcup {^C[\longrightarrow Y\rangle_{\rm epi}}$ where $C$ runs over all objects in $\mathcal{C}$, and the Auslander bijection (\ref{equ:2}) at $Y$.

Following \cite{LZ}, a morphism $f\colon X\rightarrow Y$ is \emph{projectively trivial} if ${\rm Ext}_\mathcal{C}^1(f, -)=0$. For any objects $X$ and $Y$, denote by $\mathcal{P}(X, Y)$ the subset of ${\rm Hom}_\mathcal{C}(X, Y)$ formed by projectively trivial morphisms. This gives rise to an ideal $\mathcal{P}$ of $\mathcal{C}$ and the corresponding factor category is denoted by $\underline{\mathcal{C}}$. Dually, one defines \emph{injectively trivial} morphisms and the factor category $\overline{\mathcal{C}}$. For almost split sequences, we refer to \cite{ARS}.

\begin{prop} \label{prop:REC}
Let $\mathcal{C}$ be a Hom-finite $k$-linear abelian category and let $Y$ be right epi-classified. Then we have the following statements:
\begin{enumerate}
\item[(1)] if $Y$ is indecomposable, then there is an almost split sequence $0\rightarrow K\rightarrow X\rightarrow Y\rightarrow 0$ for some objects $K$ and $X$;
\item[(2)] $\mathcal{P}(C, Y)=0$ for any object $C$.
\end{enumerate}
\end{prop}

In particular, if the abelian category $\mathcal{C}$ has right determined epimorphisms, we have $\mathcal{C}=\underline{\mathcal{C}}$.

\begin{proof}
Denote by ${\rm rad}\; {\rm End}_\mathcal{C}(Y)$ the Jacobson radical of ${\rm End}_\mathcal{C}(Y)$.  We apply (REC2) to the pair $(Y, {\rm rad}\; {\rm End}_\mathcal{C}(Y))$, and we assume that it is right represented by an epimorphism $\alpha\colon X\rightarrow Y$; moreover, we may assume that $\alpha$ is right minimal. It follows from \cite[Proposition V.1.14]{ARS} that $0\rightarrow {\rm Ker}\; \alpha \rightarrow X\rightarrow Y\rightarrow 0$ is an almost split sequence.

For (2), let $f\colon C\rightarrow Y$ be a projectively trivial morphism. Then from the definition, one infers that $f$ factors through any epimorphism $\alpha\colon X\rightarrow Y$. In particular, by (REC2) we may take $\alpha$ to be the epimorphism that is right determined by $C$ with ${\rm Im}\; {\rm Hom}_\mathcal{C}(C, \alpha)=0$. This implies that $f=0$.
\end{proof}

The dual  of Definition \ref{defn:REC} is as follows: an object $Y$ in $\mathcal{C}$ is \emph{left mono-classified} if it is right epi-classified in the opposite category $\mathcal{C}^{\rm op}$; the abelian category $\mathcal{C}$ has \emph{left determined monomorphisms} if each object is left mono-classified.

The following result is an abelian analogue of \cite[Theorem 4.2]{Kr}. The proof relies on the results in \cite{LZ}.

\begin{thm}\label{thm:1}
Let $\mathcal{C}$ be a Hom-finite $k$-linear abelian category. Then $\mathcal{C}$ has Serre duality if and only if $\mathcal{C}$ has right determined epimorphisms and left determined monomorphisms.
\end{thm}

\begin{proof}
For the ``only if" part, we assume that $\mathcal{C}$ has Serre duality with its Auslander-Reiten translation $\tau$. We only prove that $\mathcal{C}$ has right determined epimorphisms. Let $\alpha\colon X\rightarrow Y$ be an epimorphism. Denote its kernel by $K$. Then we have an exact sequence in $(\mathcal{C}^{\rm op}, k\mbox{-mod})$
$${\rm Hom}_\mathcal{C}(-, X) \stackrel{{\rm Hom}_\mathcal{C}(-, \alpha)}\longrightarrow {\rm Hom}_\mathcal{C}(-, Y) \longrightarrow {\rm Ext}^1_\mathcal{C}(-, K).
$$
By the Serre duality, ${\rm Ext}^1_\mathcal{C}(-, K)\simeq D{\rm Hom}_\mathcal{C}(\tau^{-1}(K), -)$. It follows that there is a monomorphism $F^\alpha\rightarrow D{\rm Hom}_\mathcal{C}(\tau^{-1}(K), -)$. By Lemma \ref{lem:RD} the morphism $\alpha$ is right determined by $\tau^{-1}(K)$, proving (REC1).

For (REC2), let $C$ be an object and $H\subseteq {\rm Hom}_\mathcal{C}(C, Y)$ be a $\Gamma(C)$-submodule. Consider the morphism $\theta \colon {\rm Hom}_\mathcal{C}(-, Y)
\rightarrow D{\rm Hom}_\mathcal{C}(C', -)$ with $C'\in {\rm add}\; C$ and ${\rm Im}\; \theta=F^{(C, H)}$; see Section 2. By the Serre duality, we obtain a morphism
$$\theta'\colon {\rm Hom}_\mathcal{C}(-, Y)\longrightarrow {\rm Ext}_\mathcal{C}^1(-, \tau(C))$$
with ${\rm Im}\; \theta'\simeq F^{(C, H)}$. Consider the extension $\xi\colon 0\rightarrow \tau(C)\rightarrow X\stackrel{\alpha}\rightarrow Y\rightarrow 0$ corresponding to $\theta'_Y({\rm Id}_Y)$. By the Yoneda Lemma, $\theta_T'(t)=\xi.t$, where $t\colon T\rightarrow Y$ is any morphism and $\xi.t$ denotes the pullback of $\xi$ along $t$. This implies that the following sequence is exact
$${\rm Hom}_\mathcal{C}(-, X) \stackrel{{\rm Hom}_\mathcal{C}(-, \alpha)}\longrightarrow {\rm Hom}_\mathcal{C}(-, Y) \stackrel{\theta'}\longrightarrow {\rm Ext}^1_\mathcal{C}(-, \tau(C)).
$$
Recall the isomorphism ${\rm Im}\; \theta'\simeq F^{(C, H)}$. It follows from Lemmas \ref{lem:RD} and \ref{lem:CH} that the pair $(C, H)$ is right represented by $\alpha$.

For the ``if" part, we assume that $\mathcal{C}$  has right determined epimorphisms and left determined monomorphisms. By Proposition \ref{prop:REC} and its dual, we infer that $\underline{\mathcal{C}}=\mathcal{C}=\overline{\mathcal{C}}$,  and that for any indecomposable object $Y$, there exist an almost split sequences ending at $Y$ and an almost split sequence starting at $Y$. Then $\mathcal{C}$ has Serre duality by \cite[Propositions (3.1) and (3.3)]{LZ}.
\end{proof}

\begin{rem}\label{rem:1}
Let $\mathcal{C}$ be a Hom-finite $k$-linear category having Serre duality, whose Auslander-Reiten translation is denoted by $\tau$.
\begin{enumerate}
\item[(1)] By Theorem \ref{thm:1} and Lemma \ref{lem:REC}, a morphism $\alpha\colon X\rightarrow Y$ is right determined by some object if and only if it is an epimorphism, in which case $\alpha$ is right determined by $\tau^{-1}({\rm Ker}\; \alpha)$; dually, a morphism $\beta\colon Y\rightarrow Z$ is left determined by some object if and only if it is a monomorphisms, in which case $\beta$ is left determined by $\tau({\rm Cok}\; \beta)$.
\item[(2)] We assume that $\mathcal{C}$ is not zero. Then a morphism that is not epic is not right determined by any object, and thus the category $\mathcal{C}$ does not have right determined morphisms in the sense of Definition \ref{defn:1}. By Proposition \ref{prop:Dua}, the category $\mathcal{C}$ is not a dualizing $k$-variety. However, its bounded derived category $\mathbf{D}^b(\mathcal{C})$ has Serre duality \cite{RV} and thus is a dualizing $k$-variety; see \cite[Theorem 4.2]{Kr} or \cite[Corollary 2.6]{Chen}.
\end{enumerate}
\end{rem}

\vskip 5pt

\noindent{\bf Acknowledgements}\quad The authors would like to thank Professor Henning Krause for helpful discussions on the results in this paper, and to Professor Claus Michael Ringel for enlightening discussions concerning morphisms determined by modules.

\bibliography{}

\vskip 10pt

 {\footnotesize \noindent Xiao-Wu Chen, Jue Le \\
 School of Mathematical Sciences,
  University of Science and Technology of
China, Hefei 230026, Anhui, PR China \\
Wu Wen-Tsun Key Laboratory of Mathematics, USTC, Chinese Academy of Sciences, Hefei 230026, Anhui, PR China.\\
URL: http://home.ustc.edu.cn/$^\sim$xwchen, http://staff.ustc.edu.cn/$^\sim$juele.}

\end{document}